\documentclass[11pt]{amsart}
\usepackage[letterpaper,margin=1in]{geometry}
\usepackage{amsmath,amssymb,mathtools}
\usepackage{hyperref}
\hypersetup{hidelinks,pdftitle={Discreteness, finite rank, and stability of p-frame energy minimizers},pdfauthor={Josiah Park}}
\numberwithin{equation}{section}
\newtheorem{theorem}{Theorem}[section]
\newtheorem{proposition}[theorem]{Proposition}
\newtheorem{lemma}[theorem]{Lemma}
\newtheorem{corollary}[theorem]{Corollary}
\theoremstyle{definition}

\theoremstyle{remark}

\newcommand{\R}{\mathbb R}
\newcommand{\C}{\mathbb C}
\newcommand{\N}{\mathbb N}
\newcommand{\Q}{\mathbb Q}
\newcommand{\Sph}{\mathbb S^{d-1}}
\newcommand{\RP}{\mathbb{RP}^{d-1}}
\newcommand{\Prob}{\mathcal P}
\newcommand{\Min}{\mathcal M}
\newcommand{\cE}{\mathcal E}
\newcommand{\abs}[1]{\lvert#1\rvert}
\newcommand{\norm}[1]{\lVert#1\rVert}
\DeclareMathOperator{\supp}{supp}
\DeclareMathOperator{\rank}{rank}
\DeclareMathOperator{\Hess}{Hess}
\DeclareMathOperator{\Sym}{Sym}
\DeclareMathOperator{\tr}{tr}
\DeclareMathOperator{\ord}{ord}
\DeclareMathOperator{\spanop}{span}
\DeclareMathOperator{\sgn}{sgn}
\title[Discreteness and stability of frame-energy minimizers]{Discreteness, finite rank, and stability of \texorpdfstring{\(p\)}{p}-frame energy minimizers}
\author{Josiah Park}
\address{Fields Institute}
\email{josiahp@teachx.ai}
\date{September 24, 2026}

\begin{document}
\begin{abstract}
We study probability measures minimizing the energy with kernel
$\abs{x\cdot y}^{p}$ on the real unit sphere. For every positive rational
$p=a/b$, each minimizer admits a finitely supported minimizing replacement.
If $a/b$ is reduced and $b\ge d$, every minimizer on $\mathbb S^{d-1}$ is
finitely supported. The latter assertion follows from a finite-rank
algebraic lift, Nash curve selection, and a Wronskian multiplicity bound.
For irrational exponents, we prove that a positive semidefinite restriction
of finite rank cannot have infinite compact support. We also establish
local stability and a quantitative support bound under positive-definite
Hessians at every contact of every limiting minimizer. This hypothesis is
stronger than finite support and does not prove unconditional openness of
the discreteness regime. Finally, subtracting a single suitably scaled
high even power produces a uniformly nearby kernel whose every minimizer
has finite support. We identify the atomic-purity selection of these
approximants and prove a target-centered recovery theorem, separating
finite approximation from discreteness of the limiting problem.
\end{abstract}
\maketitle

\section{Introduction}
\label{sec:intro}

For $d\ge2$ and $p>0$, consider the energy
\begin{equation}\label{eq:energy}
 I_p(\mu)=\iint_{\Sph\times\Sph}\abs{x\cdot y}^{p}\,d\mu(x)\,d\mu(y),
 \qquad \mu\in\Prob(\Sph),
\end{equation}
where $\Prob(\Sph)$ denotes the Borel probability measures on the real unit
sphere. This problem links the geometry of finite frames with the
structure of equilibrium measures. Its central difficulty is not the
existence of a minimizer, which follows from compactness, but whether
minimizing mass must concentrate on finitely many directions.

Even exponents exhibit a basic obstruction to discreteness. When
$p=2k$, the kernel is polynomial and the energy depends on a finite tensor
moment; the rotation-invariant probability measure is a minimizer.
By contrast, for $0<p<2$, every minimizing measure is carried by the lines
of an orthonormal basis \cite{EhlerOkoudjou}. The conjecture of
Bilyk, Glazyrin, Matzke, Park, and Vlasiuk asks whether every minimizer is
finitely supported whenever $p$ is not an even integer
\cite[Conjecture~1.1]{BGMPVDiscreteness}. Their work proves that such
minimizing supports have empty interior. Further exact minimizers and
classification results in tight-design regimes appear in
\cite{BGMPVOptimal}. Empty interior, however, leaves substantial room for
infinite supports.

A useful starting point is that the kernel need not be positive
semidefinite on the whole sphere, but its restriction to a minimizing
support is positive semidefinite. We exploit this restriction in two
ways. Algebraic powers constrain its rank and the geometry of its
support. Independently, strict curvature of the minimizing potential
isolates contact points and makes finiteness stable under perturbation.
The hypotheses and conclusions of these mechanisms are different and
must not be conflated.

Write
\[
 e(p)=\min_{\mu\in\Prob(\Sph)}I_p(\mu),\qquad
 \Min_p=\{\mu:I_p(\mu)=e(p)\},\qquad
 D_m=\binom{d+m-1}{m}.
\]
The rank of a kernel is the supremum of the ranks of its finite Gram
matrices. Our first result concerns all positive rational exponents,
including the even integers in its finite-replacement assertion.

\begin{theorem}[Rational exponents]\label{thm:rational}
Let $p=a/b>0$, with $a,b\in\N$ coprime.
\begin{enumerate}
\item For every $\mu\in\Min_p$, there exists $\nu\in\Min_p$ with
$\supp\nu\subseteq\supp\mu$ and
\begin{equation}\label{eq:rational-bound}
 \#\supp\nu\le R_{d,a},\qquad
 R_{d,a}=\begin{cases}D_a,&a\text{ even},\\D_{2a},&a\text{ odd}.
 \end{cases}
\end{equation}
The potentials of $\mu$ and $\nu$ agree on $\supp\mu$.
\item If an infinite compact set $S\subseteq\Sph$ carries a positive
semidefinite restriction of $\abs{x\cdot y}^{a/b}$, then $b<d$.
Consequently, if $b\ge d$, every member of $\Min_p$ has finite support.
\end{enumerate}
\end{theorem}

The second assertion is not a consequence of finite replacement alone.
Its proof passes from an arbitrary compact support to an algebraic lift;
a Nash arc in the lift yields meromorphic functions on a compact curve,
and their vanishing orders force the denominator bound. In particular,
\begin{equation}\label{eq:discreteness-set}
 \cE_d=\{p>0:p\notin2\N,\ \text{every }\mu\in\Min_p
                    \text{ has finite support}\}
\end{equation}
is dense in $(0,\infty)$. This density does not establish openness or
settle any prescribed irrational exponent.

\begin{theorem}[Irrational finite-rank rigidity]\label{thm:irrational}
Let $p>0$ be irrational and $S\subseteq\Sph$ compact. If
$\abs{x\cdot y}^{p}$ is positive semidefinite and has finite rank on
$S\times S$, then $S$ is finite. Thus for $\mu\in\Min_p$, finite support
is equivalent to finite rank of its support kernel.
\end{theorem}

The proof uses analytic curve selection in the local analytic closure of
$S$, not an unjustified assertion that $S$ itself contains a curve. The
analytic obstruction is a finite-sum theorem of Cheng, Di~Scala, and Yuan
\cite{ChengDiScalaYuan}. No finite-rank bound is asserted for an arbitrary
irrational-exponent minimizer. We also give elementary, uniform bounds on
intersections of minimizing supports with great circles.

For stability, define the potential and its whole contact set by
\[
 F_{p,\mu}(x)=\int\abs{x\cdot y}^{p}\,d\mu(y),\qquad
 C_p(\mu)=\{x\in\Sph:F_{p,\mu}(x)=e(p)\},\quad\mu\in\Min_p.
\]
The support is contained in $C_p(\mu)$, but equality is not assumed.

\begin{theorem}[Nondegenerate-contact stability]\label{thm:stability}
Fix $p_0>2$. Suppose that, for every $\nu\in\Min_{p_0}$ and every
$x\in C_{p_0}(\nu)$, the spherical Hessian
$\Hess F_{p_0,\nu}(x)$ is positive definite. Then there exist
$\delta,h>0$ and $N<\infty$ such that every $\mu\in\Min_p$ with
$\abs{p-p_0}<\delta$ satisfies
\[
 \Hess F_{p,\mu}(x)\succeq h\,\mathrm{Id}
 \quad(x\in C_p(\mu)),\qquad \#\supp\mu\le N.
\]
\end{theorem}

Section~\ref{sec:stability} makes the bound quantitative. Finiteness of a
minimizer does not imply the Hessian hypothesis. In particular, an
isolated contact may be degenerate, and a centered atom split has no
quadratic signed self-interaction. The theorem is therefore a conditional
stability result, not a proof that the set \eqref{eq:discreteness-set} is
open.

A different way to force finiteness is to flatten the kernel at the
diagonal by a small polynomial correction.

\begin{theorem}[Finite-support perturbations]\label{thm:correction}
For every $p>0$ and every even integer $q>p$, every probability measure
minimizing the continuous kernel
\begin{equation}\label{eq:corrected-kernel}
 L_{p,q}(x,y)=\abs{x\cdot y}^{p}-\frac pq\abs{x\cdot y}^{q}
\end{equation}
has finite support. Moreover,
$\norm{L_{p,q}-K_p}_\infty=p/q$, where $K_p(x,y)=\abs{x\cdot y}^{p}$.
\end{theorem}

We identify the limits of these corrected minimizers by a maximal
projective atomic-purity principle. A rotationally invariant
counterexample shows why finite support for every corrected problem need
not pass to the limit. Finally, a target-centered positive penalty
recovers any prescribed minimizer from nearby kernels. This separates
three questions: existence of a finite replacement, finiteness of every
minimizer, and a uniform support bound along an approximation. The general
non-even-exponent discreteness conjecture remains unresolved by these
results.

Throughout, weak convergence of measures is denoted by $\rightharpoonup$.
Because the kernel is even in each variable, its energy depends only on
the pushforward to $\RP$. A measure on the sphere has finite support if
and only if this pushforward does. Cardinality bounds on the sphere count
both signs; when projective points are counted instead, this is stated
explicitly.

\section{Variational structure}
\label{sec:variations}

For a continuous symmetric kernel $K$ on a compact metric space $X$, we
use $I_K(\eta,\xi)=\iint K\,d\eta\,d\xi$ for finite signed measures and
write $I_K(\eta)=I_K(\eta,\eta)$. Its energy is continuous on $\Prob(X)$,
so it attains a minimum. The following formulation records precisely
which positivity follows from minimality.

\begin{lemma}[Contact and centered positivity]\label{lem:psd}
Let $\mu$ minimize $I_K$, let $E=I_K(\mu)$, and let $S=\supp\mu$. Then
$F_\mu(x)=\int K(x,y)\,d\mu(y)\ge E$ on $X$, with equality on $S$.
For any $x_1,\ldots,x_n\in S$ and $c_1,\ldots,c_n\in\R$,
\begin{equation}\label{eq:centered-psd}
 \sum_{i,j=1}^n c_i c_j K(x_i,x_j)
       \ge E\left(\sum_{i=1}^n c_i\right)^2.
\end{equation}
Thus $K-E$ is positive semidefinite on $S$. If $E\ge0$, so is $K$.
\end{lemma}

\begin{proof}
The right derivative of $I_K((1-t)\mu+t\delta_x)$ at $t=0$ gives
$F_\mu(x)\ge E$. Since $\int F_\mu\,d\mu=E$, continuity gives equality
on the support.

To justify signed variations even when $\mu$ is nonatomic, let
$\mu_{i,\varepsilon}$ be its normalized restriction to
$B(x_i,\varepsilon)$ and set
\[
 s=\sum_i c_i,\qquad
 \eta_\varepsilon=\sum_i c_i\mu_{i,\varepsilon}-s\mu.
\]
This measure has total mass zero and density
\[
 g_\varepsilon=\frac{d\eta_\varepsilon}{d\mu}
 =\sum_i\frac{c_i\mathbf1_{B(x_i,\varepsilon)}}
                  {\mu(B(x_i,\varepsilon))}-s.
\]
All denominators are positive because $x_i\in S$, and $g_\varepsilon$
is bounded for fixed $\varepsilon$. Thus
$\mu+t\eta_\varepsilon=(1+tg_\varepsilon)\mu$ is a probability measure
for both signs of $t$ with $|t|\norm{g_\varepsilon}_\infty<1$.
Also $I_K(\mu,\eta_\varepsilon)=0$, since $F_\mu=E$ on $S$.
Minimality therefore gives $I_K(\eta_\varepsilon)\ge0$. Expanding yields
\[
 \sum_{i,j}c_ic_j I_K(\mu_{i,\varepsilon},\mu_{j,\varepsilon})
       \ge Es^2.
\]
Now $\mu_{i,\varepsilon}\rightharpoonup\delta_{x_i}$; continuity of the
kernel proves \eqref{eq:centered-psd}.
\end{proof}

For $K_p$, the minimum is strictly positive. Indeed, if
$M_\mu=\int xx^{\mathsf T}\,d\mu(x)$, then
\[
 I_2(\mu)=\tr(M_\mu^2)\ge\frac1d,
 \qquad \tr M_\mu=1.
\]
Comparing $\abs{t}^p$ with $t^2$ when $p\le2$, and applying Jensen's
inequality to $\abs{x\cdot y}^2$ when $p\ge2$, yields
\begin{equation}\label{eq:positive-minimum}
 e(p)\ge
 \begin{cases}d^{-1},&0<p\le2,\\d^{-p/2},&p\ge2.
 \end{cases}
\end{equation}
In particular, both $K_p-e(p)$ and $K_p$ are positive semidefinite on any
minimizing support. Such a support spans $\R^d$: otherwise a unit vector
orthogonal to its span would have zero potential, contradicting
\eqref{eq:positive-minimum} and the contact inequality.

\begin{proposition}[The range $0<p<2$; cf.~\cite{EhlerOkoudjou}]
\label{prop:small-p}
If $0<p<2$, every minimizing measure is supported on the $d$ lines of an
orthonormal basis, each line having total mass $1/d$. Conversely every
such measure minimizes. Its mass may be split arbitrarily between the two
unit vectors on each line.
\end{proposition}

\begin{proof}
A basis measure attains $1/d$, so equality must hold in
$I_p(\mu)\ge I_2(\mu)\ge1/d$. Equality in the first inequality implies
$\abs{x\cdot y}\in\{0,1\}$ for every support pair: the continuous
nonnegative difference of the two kernels has zero integral. Distinct
support lines are therefore orthogonal. Equality in the second inequality
is $M_\mu=d^{-1}\mathrm{Id}$, which forces exactly $d$ such lines with
equal masses. The converse is immediate.
\end{proof}

For comparison, if $p=2k$ and $\sigma$ is rotation-invariant probability,
then $\sigma$ minimizes. To see this, write
\[
 I_{2k}(\mu)=\left\|\int x^{\otimes 2k}\,d\mu(x)\right\|^2.
\]
Averaging the tensor over the orthogonal group is an orthogonal projection
and sends every such moment to the moment of $\sigma$. This proves
$I_{2k}(\mu)\ge I_{2k}(\sigma)$ and explains why even exponents must be
excluded from an every-minimizer discreteness conjecture.

The contact inequality also gives information away from the existing
support.

\begin{proposition}[A global variational bound]\label{prop:global}
Suppose in addition that $K(x,x)=1$. Let $\mu$ minimize $I_K$, and put
$E=I_K(\mu)$ and $g=F_\mu-E$. If $f\in L^\infty(\mu)$ is real,
$\int f\,d\mu=0$, and $\eta=f\mu$, then
\begin{equation}\label{eq:global-obstacle}
 \abs{F_\eta(z)}
 \le \norm f_\infty g(z)+\sqrt{(1-E)I_K(\eta)}
 \qquad(z\in X).
\end{equation}
Here $I_K(\eta)\ge0$ and $E\le1$.
\end{proposition}

\begin{proof}
The asserted signs follow from bounded-density variations and comparison
with a point mass. If $M=\norm f_\infty=0$, there is nothing to prove.
Otherwise set $Q=I_K(\eta)$, $b=F_\eta(z)$, and
$\rho=(1-\sgn(b)f/M)\mu$, using either sign when $b=0$.
Then $\rho$ is a probability measure,
\[
 I_K(\rho)=E+Q/M^2,\qquad
 I_K(\rho,\delta_z)=E+g(z)-\abs b/M.
\]
Comparison of $\mu$ with $(1-t)\rho+t\delta_z$, followed by division by
$(1-t)^2$ and substitution $r=t/(1-t)$, gives
\[
 0\le Q/M^2+2r\bigl(g(z)-\abs b/M\bigr)+r^2(1-E)
 \qquad(r\ge0).
\]
For $a,d\ge0$, a quadratic $a+2cr+dr^2$ is nonnegative on $[0,\infty)$
only if $c\ge-\sqrt{ad}$. This proves \eqref{eq:global-obstacle},
including the cases $Q=0$ or $E=1$.
\end{proof}

For example, if $I_K(\eta)=0$, then $F_\eta$ vanishes on the whole contact
set, not only on the support. This is a restriction on an existing exact
null direction; it does not prove that such a direction exists on an
infinite support.

\section{Rational exponents: rank and the denominator obstruction}
\label{sec:rational}

\subsection{Finite minimizing replacements}

\begin{lemma}[Rank under Hadamard powers]\label{lem:hadamard}
If $H$ is a positive semidefinite kernel and $m\in\N$, then
$\rank H\le\rank(H^{\circ m})$, where
$H^{\circ m}(x,y)=H(x,y)^m$.
\end{lemma}

\begin{proof}
Represent $H(x,y)=\langle h_x,h_y\rangle$ in a real Hilbert space. The
kernel $H^{\circ m}$ has feature vectors $h_x^{\otimes m}$. If
$h_{x_1},\ldots,h_{x_r}$ are linearly independent, choose linear
functionals $\ell_i$ on their span with $\ell_i(h_{x_j})=\delta_{ij}$.
Applying $\ell_i^{\otimes m}$ shows that their tensor powers are also
independent. Taking the supremum over finite independent families proves
the assertion, including infinite rank.
\end{proof}

For $p=a/b$ in lowest terms,
\begin{equation}\label{eq:power-polynomial}
 K_p^{\circ 2b}(x,y)=(x\cdot y)^{2a}.
\end{equation}
The polynomial kernel on the right has feature map $x^{\otimes2a}$ in
$\Sym^{2a}\R^d$, of dimension $D_{2a}$. When $a$ is even one may instead
use $K_p^{\circ b}=(x\cdot y)^a$. Thus every positive semidefinite
restriction of $K_p$ has rank at most $R_{d,a}$ in
\eqref{eq:rational-bound}.

\begin{proof}[Proof of Theorem~\ref{thm:rational}\textup{(1)}]
Let $S=\supp\mu$. Lemma~\ref{lem:psd} and
\eqref{eq:positive-minimum} give a continuous finite-rank Gram
representation
\[
 K_p(x,y)=\langle\Phi(x),\Phi(y)\rangle,\qquad x,y\in S,
 \qquad \Phi:S\longrightarrow\R^R,
\]
where $R\le R_{d,a}$. For completeness, choose $R$ anchors with invertible
Gram matrix $G$ and take $\Phi(x)=G^{-1/2}(K_p(x,x_i))_{i=1}^R$;
the rank condition gives the displayed identity.

Set $M=\int\Phi\,d\mu$. Then $\norm M^2=e(p)>0$ and
$\langle\Phi(x),M\rangle=e(p)$ for $x\in S$. Hence $\Phi(S)$ lies in
a proper affine hyperplane of $\R^R$. Its compact convex hull contains
$M$, so affine Carath\'eodory gives
$M=\sum_{i=1}^n w_i\Phi(x_i)$ with $n\le R$, $x_i\in S$, $w_i\ge0$,
and $\sum_iw_i=1$. The measure $\nu=\sum_iw_i\delta_{x_i}$ has
$I_p(\nu)=\norm M^2=e(p)$ and, for every $x\in S$,
$F_{p,\nu}(x)=\langle\Phi(x),M\rangle=F_{p,\mu}(x)$.
\end{proof}

\subsection{A multiplicity lemma on a compact curve}

We use two standard facts about algebraic curves. A nonconstant Nash arc
has coordinate functions algebraic over the parameter field, so its
complex Zariski closure is an irreducible curve. Normalizing the projective
closure gives a compact connected Riemann surface on which the coordinate
functions are meromorphic. Nash curve selection is available for a
nonisolated point of a real algebraic set
\cite[Proposition~8.1.13]{BochnakCosteRoy}; normalization and meromorphic
divisor theory are standard \cite{Forster}.

\begin{lemma}[Finite-dimensional vanishing orders]\label{lem:wronskian}
Let $V$ be an $r$-dimensional complex vector space of meromorphic
functions on a compact connected Riemann surface $X$. There is a finite
set $T\subset X$ such that every nonzero $f\in V$ has vanishing order at
most $r-1$ at every point of $X\setminus T$. The poles of all functions
in $V$ are bounded by a common effective divisor.
\end{lemma}

\begin{proof}
Choose a basis $f_1,\ldots,f_r$ and form its Wronskian in a local
coordinate. It is not identically zero: at a point without poles, linear
elimination gives a basis of germs with distinct initial orders; the
leading Wronskian coefficient is a nonzero Vandermonde determinant in
those orders. Under coordinate changes the Wronskian is multiplied by a
nonvanishing factor, and thus defines a meromorphic section of a tensor
power of the canonical bundle, with finitely many zeros and poles. Include these points and the poles
of the $f_i$ in $T$. Off $T$, the map assigning the first $r$ Taylor
coefficients to a member of $V$ is invertible, proving the order bound.
A common pole divisor is
\[
 D=\sum_{z\in X}\max_i\{0,-\ord_z f_i\}\,[z].
\]
\end{proof}

\subsection{From arbitrary compact support to a curve}

\begin{proof}[Proof of Theorem~\ref{thm:rational}\textup{(2)}]
Suppose $S$ is infinite and compact and $K_{a/b}|_{S\times S}$ is positive
semidefinite. By Lemma~\ref{lem:hadamard}, there is a continuous
finite-dimensional feature map $h:S\to\R^R$ with
\[
 \langle h(x),h(y)\rangle=\abs{x\cdot y}^{a/b},
 \qquad \norm{h(x)}=1.
\]
Let $A=\{(x,h(x)):x\in S\}$, and let $Z$ be its real Zariski closure,
that is, the common real zero set of all polynomials vanishing on $A$.
By the Hilbert basis theorem this is a real algebraic set. Since
$\norm x^2=\norm u^2=1$ on $A$, these equations hold on $Z$, so $Z$ is
compact. The identity
\begin{equation}\label{eq:lift-identity}
 (u\cdot v)^{2b}=(x\cdot y)^{2a}
\end{equation}
holds for all $(x,u),(y,v)\in Z$: first fix one argument in $A$ and
extend polynomially in the other, then extend in the first argument.
If $(x,u),(x,v)\in Z$, the unit norm equations and
\eqref{eq:lift-identity} give $v=\pm u$. Thus projection onto the
$x$-coordinates has fibers of size at most two.

The infinite compact algebraic set $Z$ has a nonisolated point $z_0$.
Apply Nash curve selection to the semialgebraic set $Z\setminus\{z_0\}$
at $z_0$. It gives a nonconstant Nash arc
$\gamma:[0,\varepsilon)\to Z$, with $\gamma(0)=z_0$ and
$\gamma(t)\ne z_0$ for $t>0$. Write $\gamma(t)=(x(t),u(t))$.
Its projection $x(t)$ is nonconstant, since a continuous arc in a finite
fiber is constant.

Evaluation on the complexified germ of $\gamma$ maps the complex
polynomial ring into the integral domain $\C\{t\}$, so its kernel is
prime. Since the coordinate germs are Nash,
$\C(t,\gamma_1,\ldots,\gamma_{d+R})$ is algebraic over $\C(t)$.
Its subfield generated by the coordinate germs has transcendence degree
at most one, and nonconstancy makes that degree one. The complex Zariski
closure $C$ of the arc is therefore an irreducible affine curve. The arc
is Zariski dense in $C$, and successive extension in each variable makes
$\gamma\times\gamma$ Zariski dense in $C\times C$. Thus the norm
equations and \eqref{eq:lift-identity} hold on this curve and its product.
On the compact normalization $X$ of the projective closure of $C$, the
coordinates $x_i,u_j$ are meromorphic and the same identities hold
meromorphically on $X$ and $X\times X$.

Let $W$ be the real span of the projected real arc and let
$r=\dim W\le d$. In a real orthonormal basis of $W$, write its
meromorphic coordinates as $f_1,\ldots,f_r$. They are complex-linearly
independent: a complex linear relation restricts to two real relations
on a spanning real arc. Moreover,
\[
 \sum_{i=1}^r f_i^2=1.
\]
For a regular parameter $t\in X$, define meromorphic functions of $s$ by
\[
 A_t(s)=\sum_{i=1}^r f_i(s)f_i(t),\qquad
 B_t(s)=\sum_{j=1}^R u_j(s)u_j(t).
\]
They satisfy
\begin{equation}\label{eq:divisibility-identity}
 B_t^{2b}=A_t^{2a},\qquad A_t(t)=B_t(t)=1.
\end{equation}
Since $a,b$ are coprime, taking orders in
\eqref{eq:divisibility-identity} gives
\begin{equation}\label{eq:orders-divisible}
 b\mid\ord_z A_t\qquad(z\in X).
\end{equation}

Apply Lemma~\ref{lem:wronskian} to
$V=\spanop_\C\{f_1,\ldots,f_r\}$. If $b\ge r$, then
\eqref{eq:orders-divisible} forces every zero of every $A_t$ into its
fixed finite exceptional set $T$. Their poles are bounded by the fixed
divisor $D$ of that lemma. For every $t$,
\[
 \deg(A_t)_0=\deg(A_t)_\infty\le\deg D.
\]
Thus zero multiplicities at the finitely many points of $T$ are uniformly
bounded as well. Only finitely many principal divisors $(A_t)$ can occur.
Functions with the same divisor differ by a nonzero scalar, so the projective classes
$[A_t]\in\mathbb P(V)$ take only finitely many values.

On a connected open set where all coordinates are regular, the map
$t\mapsto[A_t]$ is holomorphic and hence must be constant. Independence
of the $f_i$ would then give $f_i(t)=c_i g(t)$ for fixed constants $c_i$
and one holomorphic function $g$. The identity $\sum_i f_i(t)^2=1$
forces $g^2\sum_i c_i^2=1$, so $g$ and all the $f_i$ are constant.
Analytic continuation contradicts nonconstancy of $x(t)$. Thus
$b<r\le d$. Applying Lemma~\ref{lem:psd} to a minimizing support proves
the last assertion of the theorem.
\end{proof}

\begin{corollary}\label{cor:density}
The set $\cE_d$ in \eqref{eq:discreteness-set} is dense in $(0,\infty)$.
\end{corollary}

\begin{proof}
It contains every positive rational with reduced denominator at least
$d$. Any bounded interval contains infinitely many rationals but only
finitely many with reduced denominator less than $d$.
\end{proof}

The denominator obstruction is sharp for positive semidefinite supports,
without a minimality assumption. Let $v(t)=(\cos t,\sin t)$ and
$x(t)=v(t)^{\otimes b}\in\Sym^b\R^2\cong\R^{b+1}$. On a sufficiently
short closed interval,
\[
 x(s)\cdot x(t)=\cos^b(s-t)>0,\qquad
 \abs{x(s)\cdot x(t)}^{a/b}=\cos^a(s-t).
\]
The last kernel is the Gram kernel of $v(t)^{\otimes a}$. This gives an
infinite compact positive semidefinite support in dimension $d=b+1$.
It does not supply a minimizing measure for the sphere problem.

\section{Analytic rigidity and great-circle restrictions}
\label{sec:rigidity}

\subsection{Finite rank at an irrational exponent}

We use the following consequence of
\cite[Theorem~2.1(iii)]{ChengDiScalaYuan}. If holomorphic germs
$\chi_1,\ldots,\chi_s$ vanish at $0$, at least one is nonconstant, and
\begin{equation}\label{eq:umehara}
 \left(1+\sum_{i=1}^s\abs{\chi_i(z)}^2\right)^\alpha
       =\sum_{j=1}^R\abs{Q_j(z)}^2
\end{equation}
for finitely many holomorphic germs $Q_j$, then $\alpha\in\Q$.
Indeed, the right side belongs to the finite-sum Umehara algebra, and
hence to its field of fractions. The cited theorem applies after deleting
identically zero $\chi_i$; the expression on the left is nonconstant.
We also use Noetherianity of the ring of real-analytic germs and the
real-analytic curve-selection lemma for a semianalytic zero set; see
\cite{BierstoneMilman}. This is distinct from the Nash curve selection
used in Section~\ref{sec:rational}.

\begin{proof}[Proof of Theorem~\ref{thm:irrational}]
Suppose $S$ is infinite. Choose an accumulation point $x_0$ and let $A$
be the intersection of $S$ with a sufficiently small closed spherical cap
centered at $x_0$, so that all pairwise inner products in $A$ are positive.
Let the rank of the restricted kernel on $A$ be $r$, choose $r$ anchors
$s_1,\ldots,s_r\in A$ with invertible Gram matrix $G$, and set
\[
 h(x)=G^{-1/2}\bigl((x\cdot s_i)^p\bigr)_{i=1}^r.
\]
This defines a real-analytic map in an ambient neighborhood $U$ of $x_0$
where the relevant inner products are positive. Shrink $U$ so that
$x\cdot y>0$ for $x,y\in U$ and $S\cap U\subset A$. Finite-rank Gram
linear algebra gives
\[
 R(x,y):=(x\cdot y)^p-h(x)\cdot h(y)=0
       \qquad(x,y\in S\cap U).
\]

Let $\mathcal I$ be the ideal of real-analytic germs at $x_0$ vanishing
on $S$ in some neighborhood of $x_0$. By Noetherianity of the convergent
power-series ring $\R\{x-x_0\}$, choose germ generators
$g_1,\ldots,g_m$ and representatives on a common smaller neighborhood
$V$. Then $S\cap V\subset Z:=\{g_1=\cdots=g_m=0\}$, and $x_0$ is
nonisolated in $Z$. Apply real-analytic curve selection to the
semianalytic set $Z\setminus\{x_0\}$ at $x_0$. This gives a nonconstant
analytic arc $\gamma(t)$ through $x_0$ in $Z$ for $t\ge0$.
Its convergent series extend to a two-sided interval, and the identity
theorem keeps every $g_i(\gamma(t))$ zero there. Since
$\norm x^2-1\in\mathcal I$, the arc lies on the sphere. It is not
claimed to lie in $S$.

For each fixed $y\in S\cap V$, the germ $R(\cdot,y)$ belongs to
$\mathcal I$. It vanishes on $\gamma(t)$ initially near $0$, and then
on a common connected interval by the one-variable identity theorem.
Consequently, for each fixed $t$ in that interval,
$R(\gamma(t),\cdot)$ vanishes on $S\cap V$ and belongs to
$\mathcal I$. Applying the same argument in the second variable gives
\begin{equation}\label{eq:arc-identity}
 \bigl(\gamma(t)\cdot\gamma(u)\bigr)^p
       =h(\gamma(t))\cdot h(\gamma(u))
\end{equation}
for all sufficiently small real $t,u$. This two-step continuation avoids
any assumption of a uniform radius for ideal membership as a parameter
varies.

Rotate so that $\gamma(0)=e_1$ and extend the real-analytic functions
holomorphically. The identity theorem, applied successively in the two
parameters of \eqref{eq:arc-identity}, gives the polarized holomorphic
identity near $(0,0)$, with the power branch fixed by its value at
$(0,0)$. Substitution of the antiholomorphic diagonal $u=\bar z$ yields
\[
 \sum_{j=1}^r\abs{h_j(\gamma(z))}^2
       =\left(\sum_{i=1}^d\abs{\gamma_i(z)}^2\right)^p.
\]
The germ $\gamma_1$ is nonzero near $0$ and admits a holomorphic $p$th
power there. Divide by $\abs{\gamma_1(z)}^{2p}$ and put
\[
 \chi_i=\gamma_i/\gamma_1\quad(2\le i\le d),\qquad
 Q_j=(h_j\circ\gamma)/\gamma_1^p.
\]
All the $\chi_i$ vanish at $0$, and at least one is nonconstant;
otherwise the real unit arc would be constant. Delete the identically
zero $\chi_i$; every remaining germ is nonconstant. We obtain
\eqref{eq:umehara} with $\alpha=p$. Its nonconstant left side equals a
finite sum of squared moduli and therefore belongs to the nonconstant
part of the Umehara fraction field. Theorem~2.1(iii) of
\cite{ChengDiScalaYuan} forces $p\in\Q$, a contradiction.
The equivalence for minimizing measures now follows from
Lemma~\ref{lem:psd}, since a kernel on a finite set has finite rank.
\end{proof}

\subsection{Uniform finiteness on great circles}

A different argument bounds intersections with a great circle without
assuming finite rank. The quantitative distinction between a short-arc
bound and a whole-circle bound is important.

\begin{theorem}\label{thm:circles}
For every $p>0$ with $p\notin2\N$, there is a finite number $B(p)$ such
that every minimizing support in every dimension meets each great circle
in at most $B(p)$ points. If $p=m$ is odd, one may take $B(p)=10m$.
If $p$ is noninteger, there is $\delta(p)>0$ such that one may take
\[
 B(p)=(\lfloor p\rfloor+2)\left\lceil\frac{2\pi}{\delta(p)}\right\rceil.
\]
In particular, every minimizer on $\mathbb S^1$ has finite support.
\end{theorem}

\begin{proof}
First let $p$ be noninteger, set $a=\lfloor p\rfloor$, and put $N=a+3$.
Parametrize a short great-circle arc by
$x(u)=(z+uv)/(1+u^2)^{1/2}$, where $z,v$ are orthonormal. Its kernel
matrix is positively diagonally congruent to
\[
 G_{ij}=(1+u_i u_j)^p.
\]
For distinct nodes, divided-difference elimination in rows and columns
gives
\begin{equation}\label{eq:divided-det}
 \det G=\prod_{i<j}(u_j-u_i)^2\det H,\qquad
 H_{ij}=g[u_1,\ldots,u_i;u_1,\ldots,u_j],
\end{equation}
where $g(u,v)=(1+uv)^p$. The integral formula for divided differences
\cite[Chapter~I]{deBoor} gives, uniformly as all nodes tend to $0$,
\[
 H_{ij}\longrightarrow
 \frac{\partial_u^{i-1}\partial_v^{j-1}g(0,0)}{(i-1)!(j-1)!}
       =\delta_{ij}\binom p{i-1}.
\]
Uniformity holds regardless of how closely some nodes approach each
other. The coefficients $\binom p0,\ldots,\binom p{a+1}$ are positive
and $\binom p{a+2}<0$. Thus \eqref{eq:divided-det} is negative for $N$
distinct nodes in a sufficiently short arc, contradicting positive
semidefiniteness. Choose its angular length $\delta(p)$ uniformly by
rotation. Covering the full circle by
$\lceil2\pi/\delta(p)\rceil$ such arcs gives the stated bound.

Now let $p=m$ be odd, $S$ be a minimizing support, and $A\subset S$ a
pairwise acute subset. Put $\phi(x)=x^{\otimes m}$. From
\eqref{eq:centered-psd}, for every finite family in $A$,
\[
 \left\|\sum_i c_i\phi(x_i)\right\|^2
       \ge e(m)\left(\sum_i c_i\right)^2.
\]
Therefore the functional
$\ell(\sum_i c_i\phi(x_i))=\sum_i c_i$ is well defined and bounded on
the span of $\phi(A)$. Finite-dimensional Riesz representation yields
one homogeneous odd polynomial $P$ of degree $m$ with $P=1$ on $A$.
On a great circle, $P-1$ is a nonzero trigonometric polynomial of degree
at most $m$: it cannot vanish identically because $P$ is odd. Such a
polynomial has at most $2m$ distinct zeros. Five arcs of angular length
$2\pi/5<\pi/2$ cover the circle and are pairwise acute within each arc,
giving $10m$.
\end{proof}

This theorem controls one-dimensional sections, not the full support in
higher dimensions. A latitude circle in $\mathbb S^2$, for example,
meets each great circle in at most two points. Nor does the noninteger
proof apply to an arbitrary smooth curve: its exact diagonal congruence
to $(1+uv)^p$ uses great-circle geometry. The tensor curves at the end of
Section~\ref{sec:rational} already show that general analytic curves can
carry positive semidefinite rational-power kernels.

\section{Contact curvature and conditional stability}
\label{sec:stability}

For $p>2$ the potential is $C^2$ on the sphere. At a contact point $x$,
its Hessian in a unit tangent direction $v$ is
\begin{equation}\label{eq:hessian}
 \Hess F_{p,\mu}(x)[v,v]
 =p(p-1)\int\abs{x\cdot y}^{p-2}(v\cdot y)^2\,d\mu(y)-p e(p).
\end{equation}
This follows by differentiating along
$x(s)=x\cos s+v\sin s$. All derivatives extend continuously across
$x\cdot y=0$ when $p>2$. Since contacts are global minima of the
potential, \eqref{eq:hessian} is nonnegative. Strict positivity requires
an additional argument or hypothesis.

\subsection{The exact centered split}

Suppose $\mu=w\delta_x+\lambda$ is minimizing, $w>0$. Replacing its atom
by $w\rho$, where $\rho$ is a probability measure, gives the exact
identity
\begin{equation}\label{eq:split-exact}
 I_p\bigl(\mu+w(\rho-\delta_x)\bigr)-e(p)
 =2w\int(F_{p,\mu}-e(p))\,d\rho+w^2I_p(\rho-\delta_x).
\end{equation}
Take
$\rho_\varepsilon=(\delta_{x\cos\varepsilon+v\sin\varepsilon}
+\delta_{x\cos\varepsilon-v\sin\varepsilon})/2$ with $v\perp x$ unit.
For sufficiently small $\varepsilon$,
\begin{align}
 I_p(\rho_\varepsilon-\delta_x)
 &=\frac32+\frac12\cos^p(2\varepsilon)-2\cos^p\varepsilon\notag\\
 &=\frac{p(3p-2)}4\varepsilon^4+O(\varepsilon^6).
 \label{eq:split-quartic}
\end{align}
The quadratic self-energy decrease cancels the contribution already
present in the first-variation term. In particular, when $p>2$, the
whole variation has leading expression
\[
 w\,\Hess F_{p,\mu}(x)[v,v]\varepsilon^2+o(\varepsilon^2)
 +w^2\frac{p(3p-2)}4\varepsilon^4+O(\varepsilon^6).
\]
It cannot force strict positivity of the Hessian. Likewise, an isolated
zero of a nonnegative analytic function need not be quadratic, as
$u^4$ demonstrates.

\subsection{Stability at the whole contact set}

\begin{lemma}\label{lem:c2-limit}
If $p_n\to p_0>2$ and $\mu_n\rightharpoonup\nu$, then
$F_{p_n,\mu_n}\to F_{p_0,\nu}$ in $C^2(\Sph)$ and
$I_{p_n}(\mu_n)\to I_{p_0}(\nu)$. If $\mu_n\in\Min_{p_n}$, then
$\nu\in\Min_{p_0}$.
\end{lemma}

\begin{proof}
On a compact parameter interval in $(2,\infty)$, the kernel and its first
two spherical derivatives in the first variable are jointly continuous
in $(p,x,y)$. Uniform continuity and a finite-net argument turn weak
convergence against these continuous functions of $y$ into uniform
convergence in $x$, also for the derivatives. Energy convergence follows
by integration; comparison with each fixed competitor proves minimality
of the limit.
\end{proof}

\begin{proposition}[Local cardinality bound]\label{prop:local-stability}
Suppose $p_n\to p_0>2$, $\mu_n\in\Min_{p_n}$, and
$\mu_n\rightharpoonup\nu$. If every point of $C_{p_0}(\nu)$ has a
positive-definite contact Hessian, then this contact set is finite and
\[
 \#\supp\mu_n\le\#C_{p_0}(\nu)
\]
for all sufficiently large $n$.
\end{proposition}

\begin{proof}
The limiting contact set is compact and every contact is isolated, so
there are finitely many. Take disjoint small geodesically convex balls
around them in which the limiting Hessian remains positive definite.
Lemma~\ref{lem:c2-limit} preserves strict convexity in these balls for
large $n$. On their compact complement the limiting contact gap is
positive; uniform convergence of potentials and minimum energies
preserves that gap as well. Every support point of $\mu_n$ therefore
lies in a ball, and strict geodesic convexity allows at most one global
potential minimum in each ball.
\end{proof}

The use of \emph{all} contacts is essential to this argument. At an unused
contact $z\in C_{p_0}(\nu)$, insertion has no first-order cost:
\begin{equation}\label{eq:unused-contact}
 I_{p_0}((1-t)\nu+t\delta_z)-e(p_0)=t^2(1-e(p_0)).
\end{equation}
Curvature only at the old support does not control vanishing mass near
such a point. The evenness of the potential causes no difficulty:
contacts occur in antipodal pairs, and their Hessians are isometric.

\begin{proof}[Proof of Theorem~\ref{thm:stability}]
Consider the compact set
\[
 \{(\nu,x,v):\nu\in\Min_{p_0},\ x\in C_{p_0}(\nu),\
                   \norm v=1,\ v\perp x\}.
\]
The Hessian value on this set is continuous and, by hypothesis, strictly
positive. It therefore has a positive minimum $h_0$.
If no nearby uniform bound $h_0/2$ existed at all contacts of all
minimizers, we could choose $p_n\to p_0$, $\mu_n\in\Min_{p_n}$, contacts
$x_n$, and unit tangent directions $v_n$ violating it. Compactness and
Lemma~\ref{lem:c2-limit} give a limiting triple in the displayed set and
a Hessian value at most $h_0/2$, a contradiction.

Choose a compact interval $J\subset(2,\infty)$ containing the resulting
parameter neighborhood. Proposition~\ref{prop:packing} below, with the
common Hessian lower bound $h_0/2$, gives a common finite support bound.
\end{proof}

Thus the all-contact nondegeneracy condition is open and supplies an
open interval in $\cE_d$ whenever it holds at a non-even exponent.
The theorem does not identify finite support with that condition and
does not establish openness of $\cE_d$ itself.

\subsection{An explicit separation bound}

\begin{proposition}\label{prop:packing}
Let $J=[a,b]\subset(2,\infty)$ and put
\begin{equation}\label{eq:packing-constants}
 \beta=\min\{1,a-2\},\qquad
 C_J=b(b-1)\bigl(\max\{1,b-2\}+2\bigr)+b^2.
\end{equation}
Suppose $p\in J$, $\mu\in\Min_p$, and
$\Hess F_{p,\mu}(x)\succeq h\,\mathrm{Id}$ for all $x\in\supp\mu$,
where $h>0$. Set
\[
 r=\min\left\{1,\left(\frac{h}{2C_J}\right)^{1/\beta}\right\}.
\]
Distinct support points have spherical distance greater than $r$, and
\begin{equation}\label{eq:packing-bound}
 \#\supp\mu\le\left(1+\frac\pi r\right)^d.
\end{equation}
On $\mathbb S^2$ one also has
$\#\supp\mu\le2/(1-\cos(r/2))$.
\end{proposition}

\begin{proof}
For a unit-speed great circle $x(s)$, write $t(s)=x(s)\cdot y$.
Then $\abs t,\abs{t'}\le1$, $t''=-t$, and
\[
 \frac{d^2}{ds^2}\abs{t(s)}^p
   =p(p-1)\abs{t(s)}^{p-2}t'(s)^2-p\abs{t(s)}^p.
\]
On $[-1,1]$, the map $t\mapsto\abs t^{p-2}$ is uniformly
$\beta$-H\"older for $p\in J$, with constant at most
$\max\{1,b-2\}$; for $\abs s\le1$ the same estimate follows by
composing with $t(s)$. Also
$\abs{t'(s)^2-t'(0)^2}\le2\abs s$ and
$\abs{\abs{t(s)}^p-\abs{t(0)}^p}\le b\abs s$.
Integration against $\mu$ gives
\[
 \abs{(F_{p,\mu}\circ x)''(s)-(F_{p,\mu}\circ x)''(0)}
       \le C_J\abs s^\beta\qquad(\abs s\le1).
\]
Starting at a support point, the first derivative vanishes and the
second derivative is at least $h$. It remains at least $h/2$ until time
$r$, so another contact cannot occur at any distance in $(0,r]$.

Geodesic separation greater than $r$ gives Euclidean chordal separation
greater than $2r/\pi$. The disjoint Euclidean balls of radius $r/\pi$
centered at the support points lie in the ball of radius $1+r/\pi$;
volume comparison proves \eqref{eq:packing-bound}. On $\mathbb S^2$,
disjoint spherical caps of radius $r/2$ instead give the last bound.
\end{proof}

When $b\ge3$, the constant in \eqref{eq:packing-constants} simplifies to
$C_J=b^3$. The estimate is conditional: no positive lower bound $h$ for
all minimizers at an arbitrary non-even exponent is supplied here.

\section{Finite-support regularization at a fixed exponent}
\label{sec:regularization}

\subsection{A flat-diagonal obstruction}

\begin{lemma}\label{lem:flat-diagonal}
Let $h:[-1,1]\to\R$ be continuous and even, and $C^2$ near $1$.
Suppose
\[
 h(1)=c>0,\qquad h'(1)=0,\qquad
 h'(t)\ne0\quad(1-\delta<t<1)
\]
for some $\delta>0$. If $h(x\cdot y)$ is positive semidefinite on an
infinite compact subset of $\Sph$, these conditions give a contradiction.
\end{lemma}

\begin{proof}
Suppose distinct support points $x_n$ converge to $x$. Put
$r_n=\norm{x_n-x}$ and, after passing to a subsequence, assume
$(x_n-x)/r_n\to u$, where $\norm u=1$ and $u\perp x$.
Since $x_n\cdot x=1-r_n^2/2$, Taylor's theorem gives
$c-h(x_n\cdot x)=O(r_n^4)$. A Gram representation of the positive
semidefinite kernel and Cauchy--Schwarz imply, for every fixed support
point $y$,
\begin{equation}\label{eq:flat-feature}
 \abs{h(x_n\cdot y)-h(x\cdot y)}^2
       \le2c\bigl(c-h(x_n\cdot x)\bigr)=O(r_n^4).
\end{equation}
Now fix $y=x_j$ sufficiently close to $x$ that
$x\cdot x_j\in(1-\delta,1)$. Divide the unsquared difference in
\eqref{eq:flat-feature} by $r_n$ and let $n\to\infty$. Differentiability
near this inner product gives
$h'(x\cdot x_j)(u\cdot x_j)=0$, hence $u\cdot x_j=0$.
This holds for every sufficiently large fixed $j$. Therefore
$u\cdot(x_j-x)/r_j=0$ for those $j$, and the limit $j\to\infty$
contradicts $\norm u=1$.
\end{proof}

\begin{proof}[Proof of Theorem~\ref{thm:correction}]
Let $q>p$ be even and put
$h_q(t)=\abs t^p-(p/q)\abs t^q$. This function is nonnegative on
$[-1,1]$, and its minimum energy is therefore nonnegative. By
Lemma~\ref{lem:psd}, its kernel is positive semidefinite on the support
of every minimizer. Near $1$,
\[
 h_q(1)=1-p/q>0,\qquad h_q'(1)=0,\qquad
 h_q'(t)=pt^{p-1}(1-t^{q-p})>0\quad(0<t<1).
\]
Lemma~\ref{lem:flat-diagonal} proves finiteness of each minimizing
support. The uniform difference from $K_p$ has maximum $p/q$, attained
on the diagonal. Since $q$ is even, the correction is a polynomial
kernel of rank at most $D_q$.
\end{proof}

The correction is small in the uniform norm, but not in $C^2$ at the
diagonal: $h_q''(1)=p(p-q)$. Thus its flat-diagonal mechanism does not
persist automatically in the unmodified limiting problem.

\subsection{Which unmodified minimizers are selected?}

In this subsection work on $X=\RP$, so one atom means one unoriented
line. All kernels under consideration descend to $X$, and finite support
on $X$ is equivalent to finite support on the sphere. Define the atomic
purity
\begin{equation}\label{eq:purity}
 A(\mu)=(\mu\otimes\mu)(\Delta_X)
       =\sum_{x\in X}\mu\{x\}^2,
 \qquad a_*(p)=\max_{\mu\in\Min_p}A(\mu).
\end{equation}
The maximum exists: $A$ is upper semicontinuous by the closedness of the
diagonal and the Portmanteau theorem. On the sphere, this same quantity
is the sum of squares of the total masses on antipodal pairs, not the
sum of squares of individual signed-point masses.

Let $\nu_q$ be any minimizer of $L_{p,q}$ and let
$e_{p,q}=I_p(\nu_q)-(p/q)I_q(\nu_q)$.

\begin{proposition}[Maximal-purity selection]\label{prop:purity}
As even $q\to\infty$ with $p$ fixed, every weak cluster point of
$\nu_q$ belongs to $\Min_p$ and has purity $a_*(p)$. Moreover,
\begin{align}
 I_q(\nu_q)&\longrightarrow a_*(p),\label{eq:purity-limit}\\
 I_p(\nu_q)-e(p)&=o(q^{-1}),\label{eq:energy-rate}\\
 e_{p,q}&=e(p)-\frac pq a_*(p)+o(q^{-1}).\label{eq:minimum-rate}
\end{align}
\end{proposition}

\begin{proof}
Write $e=e(p)$ and $a_*=a_*(p)$. Optimality against any
$\mu\in\Min_p$ gives
\begin{equation}\label{eq:corrected-comparison}
 0\le I_p(\nu_q)-e
 \le\frac pq\bigl(I_q(\nu_q)-I_q(\mu)\bigr)\le\frac pq.
\end{equation}
In particular, all cluster points minimize $I_p$. Choose a
purity-maximizing $\mu_*\in\Min_p$. Since $I_q(\mu_*)\ge A(\mu_*)$,
\eqref{eq:corrected-comparison} gives $I_q(\nu_q)\ge a_*$.

Suppose $\nu_{q_j}\rightharpoonup\nu$. For every fixed $R>0$ and large
$j$, monotonicity of the powers on $[0,1]$ yields
\[
 \limsup_j I_{q_j}(\nu_{q_j})\le
 \lim_j I_R(\nu_{q_j})=I_R(\nu).
\]
As $R\to\infty$, dominated convergence gives $I_R(\nu)\downarrow
A(\nu)$ on projective space. Therefore
\[
 a_*\le\liminf_j I_{q_j}(\nu_{q_j})
 \le\limsup_j I_{q_j}(\nu_{q_j})\le A(\nu)\le a_*.
\]
Compactness proves \eqref{eq:purity-limit} for all choices of minimizers.
Also $I_q(\mu_*)\to a_*$. Applying
\eqref{eq:corrected-comparison} to this fixed $\mu_*$ proves
\eqref{eq:energy-rate}, and the definition of $e_{p,q}$ then gives
\eqref{eq:minimum-rate}.
\end{proof}

The conclusion is about $I_q(\nu_q)$, not necessarily $A(\nu_q)$.
It supplies neither a positive lower bound for $a_*(p)$ nor finiteness of
a limiting support. Even positive purity would only guarantee the
presence of an atom, not a finite number of atoms.

\subsection{A counterexample to passage of finiteness through the limit}

The following example uses the same type of high-even-power correction,
but a different rotationally invariant kernel. It isolates the failure
of the approximation inference, without contradicting the $p$-frame
conjecture.

\begin{proposition}\label{prop:diffuse-limit}
On $\RP$, the kernel $B(x,y)=\exp((x\cdot y)^2-1)$ has a unique
minimizer, namely rotation-invariant probability $\sigma$; this measure
is nonatomic. For every even $q\ge6$, every minimizer of
\[
 B_q(x,y)=B(x,y)-\frac2q\abs{x\cdot y}^q
\]
has finite support. All such minimizing measures converge weakly to
$\sigma$ as $q\to\infty$, and their support cardinalities tend to
infinity.
\end{proposition}

\begin{proof}
For any finite real signed measure $\eta$ on projective space,
\[
 I_B(\eta)=\mathrm e^{-1}\sum_{k\ge0}\frac1{k!}
       \left\|\int x^{\otimes2k}\,d\eta(x)\right\|^2.
\]
The series is uniformly convergent at the kernel level. If its value is
zero, all even polynomial moments vanish. Even polynomials separate
projective points, so Stone--Weierstrass and uniqueness of measures imply
$\eta=0$. Thus the quadratic form is strictly positive on nonzero signed
measures. Rotation invariance makes $F_{B,\sigma}$ constant; for a
probability $\mu\ne\sigma$,
\[
 I_B(\mu)-I_B(\sigma)=I_B(\mu-\sigma)>0.
\]

For even $q\ge6$, put $b_q(t)=\mathrm e^{t^2-1}-(2/q)t^q$.
It is positive on $[-1,1]$, since
$b_q\ge\mathrm e^{-1}-2/q>0$, and $b_q(1)=1-2/q>0$.
Also $b_q'(1)=0$, while
\[
 b_q'(t)=2t\bigl(\mathrm e^{t^2-1}-t^{q-2}\bigr)>0
       \qquad(0<t<1).
\]
Indeed, $(q-2)\log t\le2\log t<t^2-1$ there. Lemmas
\ref{lem:psd} and \ref{lem:flat-diagonal} prove finite support of every
$B_q$-minimizer. Uniform convergence $B_q\to B$ and uniqueness of the
$B$-minimizer force convergence to $\sigma$. Measures with at most a
fixed number of atoms form a weakly compact set, so no subsequence of
these support cardinalities can remain bounded.
\end{proof}

\section{Approximation, selection, and recovery}
\label{sec:recovery}

\subsection{Exponent perturbations and their selection rule}

On every compact positive exponent interval, the kernels are locally
Lipschitz in the exponent. If $p,q\ge\alpha>0$, then
\begin{equation}\label{eq:exponent-lipschitz}
 \norm{K_p-K_q}_\infty\le\frac{\abs{p-q}}{\mathrm e\alpha},
\end{equation}
because $\sup_{0\le t\le1}t^\alpha\abs{\log t}=1/(\mathrm e\alpha)$.
Thus $e(p)$ is locally Lipschitz, and every weak limit of exact minimizers
at exponents tending to $p$ lies in $\Min_p$. This assertion neither
preserves support cardinality nor says that every member of $\Min_p$
can be recovered in that way.

To make the latter issue precise, set
\[
 J_p(\mu)=\iint\abs{x\cdot y}^{p}\log\abs{x\cdot y}
                    \,d\mu(x)\,d\mu(y),
\]
with the integrand defined as zero at $x\cdot y=0$.

\begin{proposition}[One-sided selection]\label{prop:selection}
If $h_n\downarrow0$ and $\mu_n\in\Min_{p+h_n}$, every weak limit
minimizes $J_p$ over $\Min_p$. Limits from exponents $p-h_n$ maximize
$J_p$ over $\Min_p$. Moreover,
\[
 e'_+(p)=\min_{\Min_p}J_p,\qquad
 e'_-(p)=\max_{\Min_p}J_p.
\]
\end{proposition}

\begin{proof}
Boundedness of $t^q\abs{\log t}^2$ for $q$ in a compact positive
interval gives the expansion
\[
 I_{p+h}(\rho)=I_p(\rho)+hJ_p(\rho)+O(h^2)
\]
uniformly over probability measures. Compare $\mu_n$ with any fixed
$\nu\in\Min_p$, use $I_p(\mu_n)\ge e(p)$, divide by $h_n>0$, and
pass to a weak limit. This gives $J_p(\mu)\le J_p(\nu)$; a negative
increment reverses the inequality. For the derivative formula from the
right, comparison with a $J_p$-minimizer supplies the upper bound, while
the expansion at $\mu_n$ and $I_p(\mu_n)\ge e(p)$ supply the lower
bound. Compactness makes them agree. The left derivative is analogous.
\end{proof}

For any fixed $N$, the set $\Prob_{\le N}(X)$ of probability measures
with at most $N$ atoms is weakly compact: it is the image of the compact
space $\Delta_{N-1}\times X^N$ under
$(w,x)\mapsto\sum_iw_i\delta_{x_i}$. Therefore a uniformly bounded
sequence of exact finite minimizers proves existence of a finite
minimizing limit. Without a recovery argument, it does not prove that a
prescribed co-minimizer is finite. Also, an unbounded sequence of support
sizes need not have an infinite-support limit; atoms may coalesce or
lose mass.

\subsection{Recovering a prescribed minimizer}

The following result addresses selection, but intentionally leaves
support bounds as a separate question.

\begin{theorem}[Target-centered recovery]\label{thm:recovery}
Let $X$ be compact metric and $K_n,K$ continuous symmetric kernels with
$\delta_n=\norm{K_n-K}_\infty\to0$. Fix any minimizer $\mu$ of $I_K$.
There is a continuous positive semidefinite kernel $H_\mu$ such that,
for every $\varepsilon_n>0$ with
$\varepsilon_n\to0$ and $\delta_n/\varepsilon_n\to0$, every choice of
minimizers $\nu_n$ of $I_{K_n+\varepsilon_nH_\mu}$ satisfies
$\nu_n\rightharpoonup\mu$.
\end{theorem}

\begin{proof}
Choose a countable dense family $(\phi_j)$ in the unit ball of $C(X)$ and
let
\[
 H(x,y)=\sum_{j\ge1}2^{-j}\phi_j(x)\phi_j(y).
\]
The series converges uniformly and defines a continuous positive
semidefinite kernel. The quantity
\[
 D_H(\rho,\mu)^2=I_H(\rho-\mu)
 =\sum_{j\ge1}2^{-j}\left(\int\phi_j\,d(\rho-\mu)\right)^2
\]
vanishes only when $\rho=\mu$. Center its feature map at $\mu$:
\begin{align*}
 H_\mu(x,y)={}&H(x,y)-\int H(x,z)\,d\mu(z)
             -\int H(z,y)\,d\mu(z)+I_H(\mu).
\end{align*}
This is again positive semidefinite and satisfies
$I_{H_\mu}(\rho)=D_H(\rho,\mu)^2$ for probability measures.
If $e=I_K(\mu)$, comparison of $\nu_n$ with $\mu$ gives
\[
 e-\delta_n+\varepsilon_n D_H(\nu_n,\mu)^2
 \le I_{K_n}(\nu_n)+\varepsilon_nD_H(\nu_n,\mu)^2
 \le I_{K_n}(\mu)\le e+\delta_n.
\]
Consequently
\begin{equation}\label{eq:recovery-estimate}
 \varepsilon_nD_H(\nu_n,\mu)^2\le2\delta_n.
\end{equation}
By compactness, continuity of $H$, and separation of measures by the
$\phi_j$, every cluster point is $\mu$.
\end{proof}

\begin{corollary}\label{cor:recovery-bound}
Under the hypotheses of Theorem~\ref{thm:recovery}, if one can choose
recovering minimizers with a common bound
$\#\supp\nu_n\le N_\mu<\infty$, then
$\#\supp\mu\le N_\mu$. If this hypothesis holds for every prescribed
minimizer $\mu$, then every minimizer of $I_K$ is finitely supported.
\end{corollary}

\begin{proof}
Apply weak compactness of $\Prob_{\le N_\mu}(X)$ to the recovery theorem.
\end{proof}

For the sphere-power family, one may take $K_n=K_{p_n}$ with rational
$p_n\to p$ and, for instance,
$\varepsilon_n=\abs{p_n-p}^{1/2}$ when $p_n\ne p$.
Equation~\eqref{eq:exponent-lipschitz} verifies the required ratio.
However, Theorem~\ref{thm:rational} concerns the \emph{unmodified} rational
kernels, not the target-centered kernels in this construction. A uniform
support bound for the latter remains an additional input. Likewise,
Theorem~\ref{thm:irrational} reduces irrational discreteness to a rank
question, and Theorem~\ref{thm:stability} reduces robust local finiteness
to all-contact nondegeneracy; neither missing hypothesis follows from
finite approximation alone.

\section*{Acknowledgments}
The author acknowledges OpenAI Astra for assistance with mathematical
analysis, proof auditing, and manuscript preparation. This assistance does
not constitute independent peer review. Responsibility for the
mathematical claims and bibliographic choices remains with the author.

\bibliographystyle{amsplain}
\bibliography{references}
\end{document}